\documentclass[11pt, oneside]{article}

\usepackage{enumerate}
\usepackage{amscd}
\usepackage{amsmath}
\usepackage{epsfig}
\usepackage{amssymb}
\usepackage{amsthm}    
\usepackage{latexsym}
\usepackage{graphics}

\numberwithin{equation}{section}

\theoremstyle{plain}
\newtheorem{theorem}{Theorem}[section]
\newtheorem{proposition}[theorem]{Proposition} 
\newtheorem{lemma}[theorem]{Lemma} 
\theoremstyle{definition}
\newtheorem{definition}[theorem]{Definition} 
\newtheorem{question}[theorem]{Question} 
\theoremstyle{remark}
\newtheorem{remark}[theorem]{Remark} 
\newtheorem{example}[theorem]{Example} 
{\it}{\rm}  

\newcommand{\sss}{\scriptscriptstyle}
\newcommand{\email}[1]{{\scriptsize{\it E-mail address}\/: {\rm #1}} }
\providecommand{\norm}[1]{\lVert#1\rVert}

\DeclareMathOperator{\diag}{diag}

\makeatletter
\def\blfootnote{\gdef\@thefnmark{}\@footnotetext}
\makeatother

\begin{document}
\blfootnote{\textup{2010} \textit{Mathematics Subject Classification}.
52C25, 52B11, 51F99.}

\blfootnote{\textit{key words and phrases}.
rigidity, lifting, $k$-stress, dual, characteristic polynomial.}

\blfootnote{\email{lizhaozhang@alum.mit.edu}}

\begin{titlepage}
\title{Lifting degenerate simplices with a single volume constraint}

\author{Lizhao Zhang
}
\date{}
\end{titlepage}

\maketitle

\begin{abstract}
Let $M^d$ be the spherical, Euclidean, or hyperbolic space of dimension $d\ge n+1$.
Given any degenerate $(n+1)$-simplex $\mathbf{A}$ in $M^d$ with non-degenerate $n$-faces $F_i$,
there is a natural partition of the set of $n$-faces into two subsets $X_1$ and $X_2$
such that $\sum_{X_1}V_n(F_i)=\sum_{X_2}V_n(F_i)$,
except for a special spherical case where $X_2$ is the empty set
and $\sum_{X_1}V_n(F_i)=V_n(\mathbb{S}^n)$  instead.
For all cases, if the vertices vary smoothly in $M^d$ with a \emph{single} volume constraint 
that $\sum_{X_1}V_n(F_i)-\sum_{X_2}V_n(F_i)$ is preserved as a constant (0 or $V_n(\mathbb{S}^n)$),
we prove that if a \emph{stress} invariant $c_{n-1}(\alpha^{n-1})$ of the degenerate simplex is non-zero,
then the vertices will be confined to a lower dimensional $M^n$ for any sufficiently small motion.
This answers a question of the author and we also show that in the Euclidean case,
$c_{n-1}(\alpha^{n-1})=0$ is equivalent to the vertices of a \emph{dual} degenerate $(n+1)$-simplex
lying on an $(n-1)$-sphere in $\mathbb{R}^n$. 
\end{abstract}

\section{Introduction}
\label{section_intro}

Let $M^d$ of dimension $d\ge n+1$ be the spherical, Euclidean, or hyperbolic space of
constant curvature $\kappa$, and $\mathbf{A}$ be a
degenerate $(n+1)$-dimensional simplex in $M^d$ 
with non-degenerate $n$-faces.
By degenerate we mean that the vertices $\{A_1, \dots, A_{n+2}\}$ of $\mathbf{A}$
are confined to a lower dimensional $M^n$.
Due to the degeneracy, the convex hull of the vertices in $M^d$ is an $n$-dimensional region in $M^n$.
The $n$-faces $F_i$ of $\mathbf{A}$ form a double covering of this region
with a natural partition of the set of $n$-faces into two subsets $X_1$ and $X_2$
such that $\sum_{X_1}V_n(F_i)=\sum_{X_2}V_n(F_i)$;
except for a trivial exception in the spherical case when the vertices are not confined to any open half sphere,
then in this case $X_2$ is the empty set and
$\sum_{X_1}V_n(F_i)=V_n(\mathbb{S}^n)$ instead.
The partition can also be viewed as induced by Radon's theorem.

For all cases, if the vertices of $\mathbf{A}$ vary smoothly and are confined to $M^n$, 
then obviously we have that $\sum_{X_1}V_n(F_i)-\sum_{X_2}V_n(F_i)$ 
(here $V_n(F_i)$ is short for $V_n(F_i(t))$ when the context is clear)
is preserved as a constant (0 or $V_n(\mathbb{S}^n)$)
for any small motion.
But what about the \emph{inverse}?
Inspired by earlier work of the author \cite{Zhang:rigidity}, we ask the following question:

\begin{question}
\label{question_lifting}
If the vertices of $\mathbf{A}$ vary smoothly in $M^d$
with a \emph{single} volume constraint that
$\sum_{X_1}V_n(F_i)-\sum_{X_2}V_n(F_i)$ is preserved as a constant (0 or $V_n(\mathbb{S}^n)$),
then does this constraint confine the vertices of $\mathbf{A}$ 
to a lower dimensional $M^n$ for any sufficiently small motion?
\end{question}

In this paper we always require the $n$-faces to be non-degenerate during any motion.
We provide two opposing views on the question.
On the one hand, notice that up to congruence the degrees of freedom of $\mathbf{A}$ in $M^d$
is $(n+2)(n+1)/2$ (the number of edges),
or subtract 1 if the motion of $\mathbf{A}$ is confined to $M^n$.
So under a simple view of the excess degrees of freedom, with just a single volume constraint 
as that of Question~\ref{question_lifting}, except for some extreme cases
it is hardly expected that the answer might be affirmative.
But on the other hand, if we treat all $(n+1)$-simplices up to congruence 
as points in an $(n+2)(n+1)/2$-dimensional manifold and the degenerate $(n+1)$-simplices
as the boundary of the manifold, then in the manifold the region that satisfies 
$\sum_{X_1}V_n(F_i)-\sum_{X_2}V_n(F_i)=0$ (or  $V_n(\mathbb{S}^n)$)
is a codimension one region of the manifold and coincides with the boundary partially.
Then under this view the answer to Question~\ref{question_lifting} is expected to
be affirmative for ``almost all'' configurations, and the focus shifts to finding
\emph{when} $\mathbf{A}$ can be lifted to form a non-degenerate simplex.

\subsection{Main results}

In \cite{Zhang:rigidity} we obtained a sequence of \emph{stress} invariants 
$c_0(\alpha^0), \dots, c_{n+1}(\alpha^{n+1})$ for $\mathbf{A}$, 
where $\alpha$ is a 1-stress on $\mathbf{A}$,
and $\alpha^k$ is induced as a \emph{$k$-stress} on $\mathbf{A}$.
We denote the pair by $(\mathbf{A},\alpha)$
and reserve the notation for the rest of this paper.
The notion of $k$-stress on simplicial complexes
was first introduced by Lee~\cite{Lee:stress}
(see also Rybnikov~\cite{Rybnikov:stresses} and Tay \emph{et al.}~\cite{TWW}).
To answer Question~\ref{question_lifting}, we have the following theorem
where $c_{n-1}(\alpha^{n-1})$ plays a central role.

\begin{theorem}
\label{theorem_lifting}
\emph{(Main Theorem 1)}
If $\mathbf{A}$ varies smoothly in $M^d$ with a single volume constraint that
$\sum_{X_1}V_n(F_i)-\sum_{X_2}V_n(F_i)$ is preserved as a constant (0 or $V_n(\mathbb{S}^n)$)
and $c_{n-1}(\alpha^{n-1})\ne 0$,
then the vertices are confined to a lower dimensional $M^n$ for any sufficiently small motion.
\end{theorem}

To provide a simple geometric interpretation of $c_{n-1}(\alpha^{n-1})=0$ for the Euclidean case
(but not in general for the non-Euclidean case), 
we introduce the following notion.
For $n\ge 2$, with a restriction of $\mathbf{A}$ to $\mathbb{R}^n$, a degenerate $(n+1)$-simplex 
$\mathbf{B}$ in $\mathbb{R}^n$ with vertices $\{B_1, \dots, B_{n+2}\}$
is called a \emph{dual}
of $\mathbf{A}$, if it satisfies $\overrightarrow{A_iA_j}\cdot\overrightarrow{B_kB_l}=0$ for all distinct $i,j,k,l$.%
\footnote{
One simple example is in the 2-dimensional plane: if $A_1$, $A_2$, $A_3$ are the vertices of
a non-right triangle and $A_4$ is the \emph{orthocenter} of the triangle,
then $\mathbf{A}$ is dual to itself.
This notion of dual is induced from a more conventional notion of dual of
a convex polytope in $\mathbb{R}^n$,
where we leave the details to Section~\ref{section_dual}.
}
We can show that such $\mathbf{B}$ always exists
and is unique up to similarity.

Similarly to $\mathbf{A}$,
we obtain a sequence of stress invariants 
$c_0(\beta^0), \dots, c_{n+1}(\beta^{n+1})$ for $\mathbf{B}$, where $\beta$ is a 1-stress on $\mathbf{B}$,
and $\beta^k$ is induced as a $k$-stress on $\mathbf{B}$. 
We will show that numerically we can set $\beta=\alpha$.
Then we have the following result.

\begin{theorem}
\label{theorem_dual_sphere}
\emph{(Main Theorem 2)}
If $\mathbf{B}$ is a dual of $\mathbf{A}$ in $\mathbb{R}^n$,
then $c_{n-1}(\alpha^{n-1})=0$ if and only if  $c_1(\beta^1)=0$,
which is also equivalent to the vertices of $\mathbf{B}$ lying on 
an $(n-1)$-dimensional sphere in $\mathbb{R}^n$.
\end{theorem}

The question remains as to what can be said about the non-Euclidean case.
As a consequence of Theorem~\ref{theorem_dual_sphere} here we give a quick example.
For $n=2$, $c_1(\alpha^1)=0$ and $c_1(\beta^1)=0$ coincide,
so by Theorem~\ref{theorem_lifting} and \ref{theorem_dual_sphere},
it means that in order for four points to be lifted from $\mathbb{R}^2$ 
to form a non-degenerate 3-simplex in $\mathbb{R}^3$
while preserving $\sum_{X_1}V_2(F_i)=\sum_{X_2}V_2(F_i)$ during the motion,
they have to move on to a common circle in $\mathbb{R}^2$ first before being lifted from $\mathbb{R}^2$.
See also Example~\ref{example_four_points}.

We introduce a notion of \emph{characteristic polynomial} of $(\mathbf{A},\alpha)$ by defining
\[f(x)=\sum_{i=0}^{n+1}(-1)^i c_i(\alpha^i)x^{n+1-i}.
\]
For the Euclidean case, by \cite[Theorem~3.4]{Zhang:rigidity}
$f(x)$ has one zero and $n$ non-zero real roots.%
\footnote{But in this paper we do not need to use the property that the roots are \emph{real}.
}
Similarly we let $g(x)$ be the characteristic polynomial of $(\mathbf{B},\beta)$.
The following result shows a duality between $f(x)$ and $g(x)$.

\begin{theorem}
\label{theorem_dual_characteristic_polynomial}
\emph{(Main Theorem 3)}
If $\mathbf{B}$ is a dual of $\mathbf{A}$ in $\mathbb{R}^n$
and the non-zero roots of $f(x)$ are $\{\lambda_1,\dots,\lambda_n\}$,
then the non-zero roots of $g(x)$ are $\{c/\lambda_1,\dots,c/\lambda_n\}$
for some constant $c$.
\end{theorem}

Theorem~\ref{theorem_dual_characteristic_polynomial} proves
a main part of Theorem~\ref{theorem_dual_sphere},
namely $c_{n-1}(\alpha^{n-1})=0$ if and only if $c_1(\beta^1)=0$.
But Theorem~\ref{theorem_dual_characteristic_polynomial} is a more general result
and is of interest in its own right.
In fact, it shows that for any $i$ with $0\le i\le n$, 
$c_{n-i}(\alpha^{n-i})=0$ if and only if $c_i(\beta^i)=0$.

For the Euclidean case, combining Theorem~\ref{theorem_lifting} and \ref{theorem_dual_sphere} together,
we provide a different formulation to answer Question~\ref{question_lifting}
with a slightly stronger statement that includes continuous motion as well.
Notice that for the Euclidean case the initial value of $\sum_{X_1}V_n(F_i)-\sum_{X_2}V_n(F_i)$
is always 0.

\begin{theorem}
\label{theorem_lifting_dual_sphere}
Let $\mathbf{B}$ be a dual of $\mathbf{A}$ in $\mathbb{R}^n$
and assume the vertices of $\mathbf{B}$ not lying on an $(n-1)$-dimensional sphere in $\mathbb{R}^n$.
If $\mathbf{A}$ varies continuously in $\mathbb{R}^d$ with a single volume constraint that
$\sum_{X_1}V_n(F_i)-\sum_{X_2}V_n(F_i)$ is preserved as 0,
then the vertices of $\mathbf{A}$
are confined to a lower dimensional $\mathbb{R}^n$ for any sufficiently small motion.
\end{theorem}

While Theorem~\ref{theorem_lifting_dual_sphere} is simply a statement 
that combines Theorem~\ref{theorem_lifting} and \ref{theorem_dual_sphere} together, 
somewhat surprisingly,
an alternative elementary proof of Theorem~\ref{theorem_lifting_dual_sphere}
bypasses both Theorem~\ref{theorem_lifting} and \ref{theorem_dual_sphere}
without using $k$-stress,
and is also valid for \emph{continuous} motion as well.
Thus, combined with Theorem~\ref{theorem_dual_sphere},
it also makes Theorem~\ref{theorem_lifting} valid for continuous motion for the Euclidean case
without using any advanced tools like algebraic geometry.
The interested reader may directly read this proof in Section~\ref{section_alternative_proof}
without reading earlier sections.
But for the non-Euclidean case we do not have a similarly simple answer,
thus making $k$-stress still the main tool to solve Question~\ref{question_lifting}.

In this paper we use many results developed in \cite{Zhang:rigidity},
e.g., the deriving of the stress invariant $c_{n-1}(\alpha^{n-1})$,
or more generally $c_k(\alpha^k)$ for $0\le k\le n+1$.
While we shall not repeat all the proofs, we introduce the necessary notions 
and strive to make this note self-contained and readable independently of \cite{Zhang:rigidity}.

\subsection{Background and motivations}

A \emph{flexible polyhedron} is a closed polyhedral surface in $M^n$
that admits continuous non-rigid deformations such that all faces remain rigid.
The first example of an embedded flexible polyhedron in $\mathbb{R}^3$
was discovered by Connelly \cite{Connelly:counterexample},
and the volume of all flexible polyhedra in $\mathbb{R}^3$ was shown to remain constant 
during the continuous deformation, 
proving the \emph{bellows conjecture} formulated by Connelly and D. Sullivan
(see \cite{ConnellySabitovWalz,Sabitov:invariance}).
However, in the spherical case, Alexandrov \cite{Alexandrov:flexible}
constructed a flexible polyhedron in an open half sphere $\mathbb{S}^3_{+}$
that does not conserve the volume.
Thus the validity of the bellows conjecture depends on the curvature value of the underlying space.

If we allow the faces of a polyhedron to be non-rigid but applying volume constraints on the faces instead,
then in Question~\ref{question_lifting}
the degenerate simplex $\mathbf{A}$ can be loosely treated as a variant of flexible degenerate simplex.
But unlike the bellows conjecture, the volume rigidity result of Theorem~\ref{theorem_lifting} 
does not depend on the constant curvature value of the underlying space.

Under a similar setting as that of Question~\ref{question_lifting},
we proved a weaker version of Theorem~\ref{theorem_lifting}
in \cite[Theorem~1.4]{Zhang:rigidity} with a slight reformulation below.

\begin{theorem}
\label{theorem_lifting_n_face}
If $\mathbf{A}$ varies smoothly in $M^d$
with $n+2$ volume constraints that $V_n(F_i)$ is preserved as a constant for all $n$-faces $F_i$
and $c_{n-1}(\alpha^{n-1})\ne 0$,
then the vertices are confined to a lower dimensional $M^n$ for any sufficiently small motion.
\end{theorem}

As a special case of Theorem~\ref{theorem_lifting},
though Theorem~\ref{theorem_lifting_n_face} 
has $n+2$ volume constraints and thus
is weaker than Theorem~\ref{theorem_lifting},
it is a somewhat surprising result itself,
because $n+2$ is still far less than the degrees of freedom of $\mathbf{A}$ in $M^d$ up to congruence. 
In fact, after Theorem~\ref{theorem_lifting_n_face} was proved,
it was gradually realized that the degrees of freedom of $\mathbf{A}$
may not be the main barrier for this particular setting.
Combined with the easy fact of $\sum_{X_1}V_n(F_i)=\sum_{X_2}V_n(F_i)$ 
(or $\sum_{X_1}V_n(F_i)=V_n(\mathbb{S}^n)$ for a special spherical case),
this inspired us to move a step further and ask Question~\ref{question_lifting}, 
which lead to the formulation of Theorem~\ref{theorem_lifting}.

Our main tools to prove Theorem~\ref{theorem_lifting} are two results we developed 
in \cite{Zhang:rigidity}, Theorem~\ref{theorem_invariant_r_s_h} 
and Lemma~\ref{lemma_inequality_r_s_h}.
In Theorem~\ref{theorem_invariant_r_s_h}, for any $k$-stress $\omega$ on a cell complex in $M^d$, 
we discovered a geometric invariant $c_k(\omega)$ associated with $\omega$,
and $c_{n-1}(\alpha^{n-1})$ is obtained as a special case.
Lemma~\ref{lemma_inequality_r_s_h} is a technical result that provides a crucial estimate of the volume differential
of the $n$-faces.

To some extent, this seemingly simple volume constraint on $\mathbf{A}$ that
$\sum_{X_1}V_n(F_i)-\sum_{X_2}V_n(F_i)$ is preserved as a constant
is a rigidity property under disguise.

\section{Classification of degenerate simplices}
\label{section_classification}

We first classify the degenerate simplices based on the size of $X_1$ and $X_2$
(see Question~\ref{question_lifting}).

\begin{description}
\item[\emph{case 0.}]
One subset (say, $X_2$) is the empty set,
which can only happen in the spherical case when the vertices are not confined to any open half sphere.
This is also the only case that $\sum_{X_1}V_n(F_i)=V_n(\mathbb{S}^n)$.
All other cases have $\sum_{X_1}V_n(F_i)=\sum_{X_2}V_n(F_i)$.
\item[\emph{case 1.}]
One subset (say, $X_2$) contains exactly one $n$-face, 
which happens when one vertex $A_i$ falls in the convex hull of the remaining vertices of $\mathbf{A}$ in $M^d$.
\item[\emph{case 2.}]
For all the remaining cases, where $X_1$ and $X_2$ each contains at least two $n$-faces.
\end{description}

To prove Theorem~\ref{theorem_lifting}, 
we provide some necessary background
in Section~\ref{section_basic}--\ref{section_differential_formula}.
But for case 0 and case 1 above,
those background is not needed and the proof is elementary.
While our general proof of Theorem~\ref{theorem_lifting} covers all cases,
to illustrate the theorem in a simple setting,
we provide this elementary proof first.

The proof of Theorem~\ref{theorem_lifting} for case 1 is trivial:
Given any non-degenerate $(n+1)$-simplex in $M^d$ with $n$-faces $F_i$, 
then for any $n$-face $F_j$
we have $\sum_{i\ne j}V_n(F_i)>V_n(F_j)$, which immediately proves case 1.
Notice that this is a \emph{global} property. 

To prove case 0, it suffices to prove the following result.

\begin{theorem}
\label{theorem_sum_inequality_spherical}
For any non-degenerate $(n+1)$-simplex $\mathbf{B}$ in $\mathbb{S}^d$ with $n$-faces $F_i$,
we have $\sum V_n(F_i)<V_n(\mathbb{S}^n)$.
\end{theorem}

\begin{proof}
Let the vertices be $B_1,\dots,B_{n+2}$, and $F_i$ be the $n$-face that contains all the vertices except $B_i$.
Now let $B_{n+2}'$ be the antipodal point of $B_{n+2}$, 
and for $i<n+2$ denote by $F_i'$ the $n$-face formed by $F_i$ with vertex $B_{n+2}$ replaced by $B_{n+2}'$.

For $i<n+2$, let $C_i$ be the midpoint of the half circle with end points $B_{n+2}$ and $B_{n+2}'$
and crossing $B_i$. As $\mathbf{B}$ is non-degenerate,
then $C_1,\dots,C_{n+1}$ form a non-degenerate $n$-dimensional simplex.
Denote by $G_i$ the $(n-1)$-face formed by all the vertices $C_1,\dots,C_{n+1}$ except $C_i$.

If we treat $B_{n+2}$ and $B_{n+2}'$ as the north and south pole,
then for $i<n+2$, the $n$-dimensional region formed by the \emph{union} of $F_i$ and $F_i'$
can be cut into two regions by the upper and lower hemispheres, 
with the upper region as the \emph{join} of $G_i$ with $B_{n+2}$,
and the lower region as the join of $G_i$ with $B_{n+2}'$.
Thus 
\[V_n(F_i)+V_n(F_i')=c\cdot V_{n-1}(G_i),
\]
where the constant $c$ is $V_n(\mathbb{S}^n)/V_{n-1}(\mathbb{S}^{n-1})$.
By an induction on $n$ we have $\sum V_{n-1}(G_i)\le V_{n-1}(\mathbb{S}^{n-1})$,
where the equality holds only when $n=1$
(but we do not need to use the \emph{strict} inequality here),
thus
\begin{equation}
\label{equation_sum_F}
\sum_{i<n+2}(V_n(F_i)+V_n(F_i'))=c\cdot\sum V_{n-1}(G_i)
\le c\cdot V_{n-1}(\mathbb{S}^{n-1})=V_n(\mathbb{S}^n).
\end{equation}
As $B_{n+2}'$ and $B_1,\dots,B_{n+1}$ form a \emph{non}-degenerate $(n+1)$-simplex,
therefore 
\[\sum_{i<n+2}V_n(F_i')>V_n(F_{n+2}).
\]
Plug it in (\ref{equation_sum_F}), then
\[V_n(\mathbb{S}^n)>V_n(F_{n+2}) + \sum_{i<n+2}V_n(F_i)=\sum V_n(F_i),
\]
which finishes the proof. 
\end{proof}

Notice that Theorem~\ref{theorem_sum_inequality_spherical} is a global property as well.
However both case 0 and case 1 seem more like isolated extreme cases,
and the proof above does not indicate how to prove case 2.
In fact we can show that,
unlike case 0 and case 1 where $c_{n-1}(\alpha^{n-1})$  is always non-zero,
for case 2 $c_{n-1}(\alpha^{n-1})$  can take values positive, negative or zero as well,
and its sign depends not only on $X_1$ and $X_2$ but also on the geometric shape of $\mathbf{A}$.
In other words, case 2 is more complicated.

We next provide the background that is needed to prove case 2.

\section{Basic notions}
\label{section_basic}
As the \emph{linearity} between points in the hyperbolic space $\mathbb{H}^d$
plays an important role in this paper, 
we use the \emph{hyperboloid model} to describe $\mathbb{H}^d$ throughout the paper.
Let $\mathbb{R}^{d,1}$ be a $(d+1)$-dimensional vector space endowed with a metric
$x\cdot x=-x_0^2+x_1^2+\cdots+x_d^2$, then $\mathbb{H}^d$ is defined by
\[ \{x \in \mathbb{R}^{d,1}: x\cdot x = -1, \quad x_0 > 0\},
\]
which is the upper sheet of a two-sheeted hyperboloid.
Also let the spherical space $\mathbb{S}^d$ be the standard unit sphere centered at the origin
in $\mathbb{R}^{d+1}$. 

As it is assumed that all $n$-faces of $\mathbf{A}$ are non-degenerate, so up to a constant factor,
there is an unique sequence of non-zero coefficients
$\alpha_{1}$, \ldots, $\alpha_{n+2}\in \mathbb{R}$,
such that
\begin{equation}
\label{equation_alpha_r_s_h}
\begin{split}
&\sum\alpha_{i}A_i=0 \quad\text{and}\quad \sum\alpha_{i}=0 
\quad\text{(affine dependence for $\mathbb{R}^d$)},  \\
&\sum\alpha_{i}A_i=0 \quad\text{(linear dependence for $\mathbb{S}^d$ or $\mathbb{H}^d$)}.  \\
\end{split}
\end{equation}
We call $\alpha:=\{\alpha_1, \dots, \alpha_{n+2}\}$ a \emph{$1$-stress} on $\mathbf{A}$,
and denote the pair by $(\mathbf{A},\alpha)$ and reserve the notation for the rest of this paper.

\subsection{$k$-stress on cell complex}
\label{section_stress}

The notion of $k$-stress plays an important role in our results.
While in this paper we are only concerned with $k$-stresses 
on the boundary complex of a degenerate simplex in $M^d$,
we introduce the notion in the general setting on cell complexes (not necessarily simplicial) in $M^d$.

By a $k$-dimensional \emph{convex polytope} in $M^k$
we mean a compact subset which can be
expressed as a finite intersection of closed half spaces.
In the spherical case
for convenience we also require the convex polytope 
to be confined to an open half sphere,
so a half circle or $\mathbb{S}^0$ is not considered as a convex polytope in this context.
A \emph{cell complex} in $M^d$ is a finite set 
of convex polytopes (called \emph{cells}) in $M^d$,
such that every face (empty set included)
of a cell is also a cell in the set,
and any two cells share a unique maximal common face.
However we do not worry about overlapping or intersection
between cells in $M^d$ caused by the embedding.

If $K$ is a cell complex in $M^d$, for convenience we denote by $K$ as well the set of all its cells,
and by $K^r$ the subset of its $r$-cells.

\begin{definition}
\label{defintion_k_stress}
Consider a cell complex $K$ (not necessarily
of dimension $d-1$ or $d$) in $M^d$.
A \emph{$k$-stress} $(2\le k\le d+1)$ on $K$ is
a real-valued function $\omega$ on the $(k-1)$-cells of $K$,
such that for each $(k-2)$-cell $F$ of $K$,
\[\sum_{G\in K^{k-1},F\subset G}\omega(G)u_{\sss F,G}=0,
\]
where the sum is taken over all $(k-1)$-cells $G$ of $K$
that contain $F$,
and $u_{\sss F,G}$ is the inward unit normal to $G$ at its facet $F$.
For $k=1$, a $1$-stress is an affine dependence among the vertices
for the Euclidean case, or a linear dependence for the non-Euclidean case.
\end{definition}

The notion of (\emph{affine} and \emph{linear}) $k$-stresses was first introduced by Lee~\cite{Lee:stress}
on simplicial complexes with vertices chosen in the Euclidean space.
The notion was introduced partly under the inspiration of Kalai's proof \cite{Kalai:rigidity}
of the \emph{Lower Bound Theorem} which used classical stress.
McMullen \cite{McMullen:weights} also considered \emph{weights} on simple polytopes,
a notion dual to $k$-stresses. 
Both $k$-stresses and weights were alternative approaches to proving the $g$-theorem
for simplicial convex polytopes, whose original proof of the necessity part
by Stanley~\cite{Stanley:number}
used deep techniques from algebraic geometry.
Lee~\cite{Lee:stress} showed that
the $g$-conjecture for simplicial spheres, which remains open,
can be proved true if one can show that
for a simplicial $(d-1)$-sphere $\Delta$ with vertices chosen generically in $\mathbb{R}^d$,
the dimension of the space of affine $k$-stresses on $\Delta$ is $g_k$ for $k\le\lfloor d/2\rfloor$,
where $(g_0, g_1,\dots)$ is the $g$-vector of $\Delta$.
See, e.g., \cite{Stanley:number} for the definition of $g_k$.

Rybnikov~\cite{Rybnikov:stresses} provided a geometric variation of the notion of (affine) $k$-stress, 
extending it to cell-complexes in both Euclidean and spherical spaces.
Our notion agrees with this notion.

\subsection{$k$-stress on $\mathbf{A}$}

If $F$ is a $k$-simplex in $\mathbb{S}^d$ or $\mathbb{H}^d$ 
(which is embedded in $\mathbb{R}^{d+1}$ or $\mathbb{R}^{d,1}$ respectively)
and $B_1$, \dots, $B_{k+1}$ are the vertices, for convenience we introduce a new notation
\[\norm{F}:=|\det(B_i\cdot B_j)_{1\leq i,j\leq k+1}|^{1/2}.
\]
For the spherical case it is
$(k+1)!$ times the volume of the \emph{Euclidean} $(k+1)$-simplex
whose vertices are $O$, $B_1$, \dots, $B_{k+1}$,
and for the hyperbolic case the pseudo-volume.

\begin{definition}
\label{definition_k_stress_alpha}
Let $(\mathbf{A},\alpha)$ be as in (\ref{equation_alpha_r_s_h})
where $\alpha$ is a 1-stress on $\mathbf{A}$.
For a given $k$ $(0\leq k\leq n)$ and each
simplicial $k$-face $F$ of $\mathbf{A}$,
define a $(k+1)$-stress $\alpha^{k+1}$ on $\mathbf{A}$
by $\alpha^{k+1}(F):=(\prod_{A_s\in F}\alpha_s)k!V_k(F)$
for the Euclidean case, and
$\alpha^{k+1}(F):=(\prod_{A_s\in F}\alpha_s)\norm{F}$
for the non-Euclidean case.
\end{definition}

\begin{remark}
For notational reasons that due to the slight difference between Lee's and our notion of $(k+1)$-stresses,
we use $\alpha^{k+1}$ to denote the $(k+1)$-stress
obtained by multiplying $\alpha$ with itself for $k+1$ times and then \emph{normalized} by a
volume factor, rather than taking the value of $\prod_{A_s\in F}\alpha_s$ directly.
With the volume interpretation of $\norm{F}$ above,
it is not hard to verify that $\alpha^{k+1}$ is indeed a valid $(k+1)$-stress on $\mathbf{A}$.
\end{remark}

\section{A geometric invariant of $k$-stress}
\label{section_invariant}

As shown in Theorem~\ref{theorem_lifting}, $c_{n-1}(\alpha^{n-1})$ plays a central role
in the answer to Question~\ref{question_lifting}.
For completeness we provide the detail about how a geometric invariant $c_k(\omega)$ 
is obtained for any $k$-stress $\omega$ on a cell complex $K$.

First consider a $k$-dimensional convex polytope $F$ and any two points $P$ and $Q$ in $M^d$
in general position with respect to $F$,
and denote by $\widehat{F}$ the $(k+2)$-dimensional convex polytope in $M^d$ which is
the \emph{join} of $F$ with the segment $PQ$
(e.g., if $F$ is a $k$-simplex, then $\widehat{F}$ is a $(k+2)$-simplex).
Also let $\theta_F$ be the dihedral angle of $\widehat{F}$ at face $F$.
If $\widehat{F}$ is non-degenerate, then $\theta_F$ can vary in
such a manner that the distances between any pair of vertices of $\widehat{F}$
are fixed except between $P$ and $Q$.
It follows that $V_{k+2}(\widehat{F})$ can be treated as a function of a 
\emph{single} variable $\theta_F$, and we write the differential as
$dV_{k+2}(\widehat{F})/d\theta_F$.
Some degeneracy is allowed  and 
$\widehat{F}$ need not be a convex polytope in the strict sense,
as long as $V_{k+2}(\widehat{F})$ and $\theta_F$ can be properly defined.

We introduce the following definition.

\begin{definition}
\label{definition_partial_derivative_g_theta}
Let $F$ be a $k$-dimensional convex polytope in $M^d$
and $\widehat{F}$, $\theta_F$
be as above. If $\theta_F$ varies while all edge lengths of $\widehat{F}$ are fixed
except between $P$ and $Q$,
then define $g_{\sss F}: M^d\times M^d \rightarrow \mathbb{R}$ by
\begin{equation}
\label{equation_partial_derivative_g_theta}
g_{\sss F}(P,Q):= (k+2)!\,\frac{d V_{k+2}(\widehat{F})}{d\theta_F}.
\end{equation}
Also set $g_{\sss\varnothing}(P,Q)=1$.
\end{definition}

\begin{remark}
\label{remark_g}
When $F$ is a single point $B$, it is not hard to verify that 
for the Euclidean case we have 
$g_{\sss B}(P,Q)
=\norm{\overrightarrow{PB}}\cdot\norm{\overrightarrow{QB}}\cdot\cos\theta_B
=\overrightarrow{PB}\cdot\overrightarrow{QB}$.
For the non-Euclidean case by \cite[Corollary~2.12]{Zhang:rigidity}
\[g_{\sss B}(P,Q)
=\frac{2}{1+\kappa P\cdot Q}\overrightarrow{PB}\cdot\overrightarrow{QB}.
\]
So $g_{\sss B}$ is also \emph{approximately} the Riemannian metric at $B$
as $g_{\sss B}\sim\overrightarrow{PB}\cdot\overrightarrow{QB}$
when $P,Q\rightarrow B$,
but $g_{\sss B}$  is defined globally on $M^d$ instead of just
locally on the tangent space at $B$ as the standard Riemannian metric is.
In fact $g_{\sss B}$ is also a \emph{positive definite kernel} on $\mathbb{H}^d$ for $d=1$
\cite[Theorem~2.25]{Zhang:rigidity},
which to our knowledge is a new member to the family of known positive definite kernels,
and we conjecture for $d\ge 2$ as well.
\end{remark}

Using the Schl\"{a}fli differential formula as the main tool,
we obtained a Schl\"{a}fli differential formula for simplices based on edge lengths
\cite[Proposition~2.11]{Zhang:rigidity}
and used it to prove the following key result.
See Milnor~\cite{Milnor:Schlafli} for the description of the formula,
see also Rivin and Schlenker~\cite{RivinSchlenker} and Su{\'a}rez-Peir{\'o}~\cite{Suarez:deSitter}.

\begin{theorem}
\emph{(\cite[Theorem~2.13]{Zhang:rigidity})}
\label{theorem_invariant_r_s_h}
Let $K$ be a cell complex in $M^d$ of constant curvature $\kappa$
and $\omega$ be a $(k+1)$-stress on $k$-faces of $K$ for $k\ge 0$. 
Then as long as $g_{\sss F}(P,Q)$ is properly defined for each $F\in K^k$,
we have that
\begin{equation}
\label{equation_invariant_r_s_h}
c_{k+1}(\omega) := \sum_{F\in K^k}\omega(F)\,g_{\sss F}(P,Q) 
\end{equation}
is an invariant independent of the choice of points $P,Q\in M^d$,
and for the non-Euclidean case
\begin{equation}
\label{equation_invariant_s_h}
c_{k+1}(\omega) =\kappa (k+2)k!
\sum_{F\in K^k}\omega(F)\,V_k(F).
\end{equation}
\end{theorem}

Particularly for $(\mathbf{A},\alpha)$, we have the following definition.

\begin{definition}
\label{defintion_invariant_c_k}
Let $(\mathbf{A},\alpha)$ be as in (\ref{equation_alpha_r_s_h})
where $\alpha$ is a 1-stress on $\mathbf{A}$,
and $\alpha^{k+1}$ be the $(k+1)$-stress on $\mathbf{A}$
as in Definition~\ref{definition_k_stress_alpha}.
Then by Theorem~\ref{theorem_invariant_r_s_h} we define 
a sequence of invariants
$c_1(\alpha^1), \dots, c_{n+1}(\alpha^{n+1})$ for $(\mathbf{A},\alpha)$
(also set $c_0(\alpha^0)=1$),
and for the non-Euclidean case 
\begin{equation}
\label{equation_invariant_s_h_A}
c_{k+1}(\alpha^{k+1}) =\kappa (k+2)k!
\sum_{F\subset \mathbf{A},\dim(F)=k}\left(\prod_{A_s\in F}\alpha_s\right)
\norm{F}\,V_k(F).
\end{equation}
\end{definition}

\begin{remark}
\label{remark_invariant_n_plus_1}
For the non-Euclidean case, by (\ref{equation_invariant_s_h_A})
$c_{n+1}(\alpha^{n+1})$ vanishes unless
$\mathbf{A}$ is not confined to any open half sphere in the spherical case (case 0),
and $c_{n+1}(\alpha^{n+1})$ also vanishes for the Euclidean case
as a limit of the spherical case.
However for case 0, we can set $\alpha$ such that $\alpha_i>0$ for all $i$, 
then all $c_{k+1}(\alpha^{k+1})$ are positive, including $c_{n+1}(\alpha^{n+1})$.
\end{remark}

\section{A differential formula}
\label{section_differential_formula}

Here we provide a differential formula in Lemma~\ref{lemma_inequality_r_s_h},
a crucial estimate of the volume differential of the $n$-faces and
an important step for proving Theorem~\ref{theorem_lifting}.
Denote by $\mathbf{A}(t)$ the smooth motion of $\mathbf{A}$ in $M^d$,
and $\mathbf{A}(0)=\mathbf{A}$ the initial position.

Let $A_0(t)$ in $M^d$ be the mirror reflection of $A_1(t)$ through
a lower dimensional $M^n$ that contains points $A_2(t),\ldots,A_{n+2}(t)$.
If $A_0(t)\ne A_1(t)$, 
then $\overrightarrow{A_0A_1}$ (short for $\overrightarrow{A_0(t)A_1(t)}$)
is twice the altitude vector for $A_1(t)$ with respect to the 
linear (resp. affine) span of $A_i(t)$ of $i\ge 2$
for the non-Euclidean (resp. Euclidean) case.
It is not hard to see that if $\mathbf{A}(t)$ varies smoothly over $t$,
then $A_0(t)$ varies smoothly as well,
thus $\overrightarrow{A_0A_1}^2$ also varies smoothly.

For $t\ge 0$, $\alpha_i$ can be extended to a continuous function $\alpha_i(t)$ 
with $\alpha_{i}(0)=\alpha_{i}$ (and additionally $\sum_{i\ge 1}\alpha_i(t)=0$ for the Euclidean case),
such that
$\sum_{i\ge 1}\alpha_i(t)A_i(t)$ is a multiple of $\overrightarrow{A_0A_1}$.
Denote $\{\alpha_1(t),\dots,\alpha_{n+2}(t)\}$ by $\alpha_t$.
For a fixed $t$, $\alpha_t$ is unique up to a constant factor.
We have the following formula.

\begin{lemma}
\emph{(\cite[Proposition~2.19]{Zhang:rigidity})}
\label{lemma_inequality_r_s_h}
Let $\mathbf{A}(t)$, $\alpha_t$, and $A_0(t)$  be as above.
Assume $\mathbf{A}(t)$ varies smoothly for $t\ge 0$ in $M^d$.
If $c_{k-1}(\alpha^{k-1})\ne 0$ and 
both $\overrightarrow{A_0A_1}^2$ and $(\overrightarrow{A_0A_1}^2)'$ 
are strictly increasing for small $t\ge 0$, 
then for the non-Euclidean case
\begin{equation}
\label{equation_inequality_s_h}
2\cdot k! \sum_{\substack{G\subset \mathbf{A}(t)\\ \dim(G)=k}}
\left(\prod_{A_s(t)\in G}\alpha_s(t)\right)
\,\norm{G}\,dV_k(G)
\sim -\frac{1}{4}\alpha_1^2 c_{k-1}(\alpha^{k-1})\,d\overrightarrow{A_0A_1}^2,
\end{equation}
and for the Euclidean case
\begin{equation*}
2\cdot (k!)^2\sum_{\substack{G\subset \mathbf{A}(t)\\ \dim(G)=k}}
\left(\prod_{A_s(t)\in G}\alpha_s(t)\right)
\,V_k(G)\,dV_k(G)
\sim -\frac{1}{4}\alpha_1^2 c_{k-1}(\alpha^{k-1})\,d\overrightarrow{A_0A_1}^2.
\end{equation*}
\end{lemma}

\begin{remark}
Here the notation ``$\sim$'' means that if the two sides of the formula above
are written as $f_1(t)dt$ and $f_2(t)dt$ instead,
then $f_1(t)-f_2(t)=o(f_2(t))$ as $t\rightarrow 0$.
\end{remark}

For the purpose of this paper, we only need the formula for case $k=n$,
and the proof of Theorem~\ref{theorem_lifting} essentially follows from Lemma~\ref{lemma_inequality_r_s_h}.

\section{Proof of Theorem~\ref{theorem_lifting}}
\label{section_proof_lifting}

\begin{theorem}
\emph{(Theorem~\ref{theorem_lifting})}
If $\mathbf{A}$ varies smoothly in $M^d$ with a single volume constraint that
$\sum_{X_1}V_n(F_i)-\sum_{X_2}V_n(F_i)$ is preserved as a constant (0 or $V_n(\mathbb{S}^n)$)
and $c_{n-1}(\alpha^{n-1})\ne 0$,
then the vertices are confined to a lower dimensional $M^n$ for any sufficiently small motion.
\end{theorem}

We use the same notations as in Section~\ref{section_differential_formula},
including  $\mathbf{A}(t)$, $\alpha_t$, and $A_0(t)$.
Here we only provide the proof for the non-Euclidean case, as the Euclidean case can be treated similarly,
and both as a consequence of Lemma~\ref{lemma_inequality_r_s_h}.

\begin{proof}
As $\mathbf{A}(t)$ is smooth, both $\overrightarrow{A_0A_1}^2$ and
$(\overrightarrow{A_0A_1}^2)'=2\overrightarrow{A_0A_1}\cdot(\overrightarrow{A_0A_1})'$
are 0 at $t=0$.
If the vertices are not confined to a lower dimensional $M^n$ for some small motion,
without loss of generality we assume that both $\overrightarrow{A_0A_1}^2$
and $(\overrightarrow{A_0A_1}^2)'$ are strictly increasing for small $t\ge 0$.

Let $\alpha_1(t):=\norm{F_1(t)}$ where $F_i(t)$ are the $n$-faces of  $\mathbf{A}(t)$.
As $\alpha_t$ is unique up to a constant factor and now $\alpha_1(t)$ is fixed, 
so $\alpha_t$ is also fixed for all small $t\ge 0$.
In fact, for $i\ge 2$ if $\theta_i$ (short for $\theta_i(t)$) is the dihedral angle between
$n$-faces $F_i(t)$ and $F_1(t)$, then $\alpha_i(t)=-\norm{F_i(t)}\cos\theta_i$.
Notice that $\theta_i$ is in the neighborhood of either $0$ or $\pi$.
As a convention also set $\theta_1=\pi$.
Assume $c_{n-1}(\alpha^{n-1})\ne 0$ in the following.

In Lemma~\ref{lemma_inequality_r_s_h},
on the left side of (\ref{equation_inequality_s_h}) take $k=n$, factoring out 
$2\cdot n! \prod_{i\ge 1}\alpha_i(t)$ and replacing $G$ with an $n$-face $F_i(t)$ of $\mathbf{A}(t)$,
we have the coefficient of $dV_n(F_i)$ as $\norm{F_i(t)}/\alpha_i(t)$. Namely for small $t\ge 0$,
\begin{equation}
\label{equation_inequality_s_h_n}
\sum\frac{\norm{F_i(t)}}{\alpha_i(t)}\,dV_n(F_i)
\sim c\cdot c_{n-1}(\alpha^{n-1})\,d\overrightarrow{A_0A_1}^2
\end{equation}
for a non-zero constant $c$.

As mentioned above we have $\alpha_i(t)=-\norm{F_i(t)}\cos\theta_i$ 
(including $i=1$ where $\theta_1=\pi$ and $\alpha_1(t)=\norm{F_1(t)}$), thus
\begin{equation}
\label{equation_inequality_cos_theta}
\sum-\frac{1}{\cos\theta_i}\,dV_n(F_i)
\sim c\cdot c_{n-1}(\alpha^{n-1})\,d\overrightarrow{A_0A_1}^2.
\end{equation}

As $\overrightarrow{A_0A_1}$ is twice the altitude vector for $A_1(t)$
with respect to the linear span of $F_1(t)$, we have
$\sin^2\theta_i=O(\overrightarrow{A_0A_1}^2)$ for all $i\ge 1$, and thus
\[(1+\cos\theta_i)(1-\cos\theta_i)=\sin^2\theta_i=O(\overrightarrow{A_0A_1}^2).
\]
For each $i\ge 1$, with a properly chosen sign of $\pm 1$ depending only on whether $F_i$ is in $X_1$ or $X_2$,
we have
\begin{equation}
\label{equation_inequality_cos_theta_1}
\cos\theta_i\pm 1=O(\overrightarrow{A_0A_1}^2).
\end{equation}

Recall that at the beginning of the proof, 
we assume that both $\overrightarrow{A_0A_1}^2$ and $(\overrightarrow{A_0A_1}^2)'$
are strictly increasing for small $t\ge 0$.
Denote $\overrightarrow{A_0A_1}^2$ by $f_0(t)$, then 
\[f_0(t)=\int_0^t f_0'(t)dt\le t\cdot f_0'(t).
\]

Thus $f_0(t)/f_0'(t)\rightarrow 0$ as $t\rightarrow 0$.
As the right side of (\ref{equation_inequality_cos_theta}) 
is in the order of $f_0'(t)dt$ for small $t>0$,
so on the left side any change in the coefficients in the order of $O(f_0(t))$ can be ignored.
Then on the left side of (\ref{equation_inequality_cos_theta})
replacing $\cos\theta_i$ with a proper $\pm 1$ from (\ref{equation_inequality_cos_theta_1}),
for small $t\ge 0$

\begin{equation}
\label{equation_inequality_sum}
\sum_{X_1}dV_n(F_i)-\sum_{X_2}dV_n(F_i)
\sim c\cdot c_{n-1}(\alpha^{n-1})\,d\overrightarrow{A_0A_1}^2.
\end{equation}
This contradicts the assumption that
$\sum_{X_1}V_n(F_i)-\sum_{X_2}V_n(F_i)$ is preserved as a constant.
Thus the vertices are confined to a lower dimensional $M^n$ for small $t\ge 0$,
and this completes the proof.
\end{proof}

When $c_{n-1}(\alpha^{n-1})\ne 0$, 
(\ref{equation_inequality_sum}) implies that --- if we ignore the smoothness
requirement for a moment --- for any \emph{non}-degenerate $(n+1)$-simplex
in a small neighborhood of $\mathbf{A}$,  we always have 
$\sum_{X_1}V_n(F_i)\ne\sum_{X_2}V_n(F_i)$
(or $\sum_{X_1}V_n(F_i)\ne V_n(\mathbb{S}^n)$ for case 0),
and the strict inequality is fixed as either ``$>$'' or ``$<$''
that only depends on the sign of $c_{n-1}(\alpha^{n-1})$.

\begin{remark}
For case 2, 
further computation shows that even with fixed $X_1$ and $X_2$, 
$c_{n-1}(\alpha^{n-1})$ can take values positive, negative or zero as well; and
starting with any configuration with a non-zero $c_{n-1}(\alpha^{n-1})$,
through degenerate $(n+1)$-simplices only,
up to congruence
it can deform to any configuration with the same $X_1$ and $X_2$
and a zero $c_{n-1}(\alpha^{n-1})$,
and from there can be lifted to form a non-degenerate simplex.
Thus Theorem~\ref{theorem_lifting} for 
case 2 is a \emph{local} property, 
and cannot be strengthened by replacing the statement
``for any sufficiently small motion'' with ``for any motion''.
\end{remark}

\section{Geometric interpretations}
\label{section_geometric_interpretations}

As shown in Theorem~\ref{theorem_lifting}, $c_{n-1}(\alpha^{n-1})=0$ 
is the critical position that
$\mathbf{A}$ may be lifted from $M^n$ to form a non-degenerate $(n+1)$-simplex.
The main purpose of this section is to prove 
Theorem~\ref{theorem_dual_sphere} and \ref{theorem_dual_characteristic_polynomial},
which provide a simple geometric interpretation of $c_{n-1}(\alpha^{n-1})=0$ for the Euclidean case.
The main idea is to use matrix theory to prove Theorem~\ref{theorem_dual_characteristic_polynomial} first,
and then prove Theorem~\ref{theorem_dual_sphere} next.
For the non-Euclidean case, while an explicit formula for $c_{n-1}(\alpha^{n-1})$ is provided
in (\ref{equation_invariant_s_h_A}),
we lack a nice geometric interpretation of $c_{n-1}(\alpha^{n-1})=0$ except for $n=2$.

We first revisit some notions.

\subsection{Dual of $\mathbf{A}$ in $\mathbb{R}^n$}
\label{section_dual}

For $n\ge 2$, with a restriction of $\mathbf{A}$ to $\mathbb{R}^n$, a degenerate $(n+1)$-simplex 
$\mathbf{B}$ in $\mathbb{R}^n$ with vertices $\{B_1, \dots, B_{n+2}\}$
is called a \emph{dual} of $\mathbf{A}$,
if it satisfies $\overrightarrow{A_iA_j}\cdot\overrightarrow{B_kB_l}=0$ for all distinct $i,j,k,l$.
In the following we show that such $\mathbf{B}$ always exists
and is unique up to similarity.

Without loss of generality, let $A_{n+2}$ be the origin $O$ in $\mathbb{R}^n$, 
and $F_i$ be the $n$-face of $\mathbf{A}$ that does not contain the vertex $A_i$.
Fix a non-zero constant $c$.
For any $i\le n+1$, denote by $G_i$ the $(n-1)$-face $F_{n+2}\setminus\{A_i\}$.
Then there is a unique point $B_i$ in $\mathbb{R}^n$, 
such that $\overrightarrow{OB_i}$ is perpendicular to $G_i$,
and for any $j\le n+1$ with $j\ne i$, we have 
$\overrightarrow{OB_i}\cdot\overrightarrow{OA_j}=c$.
Finally let $B_{n+2}=O$ and $\mathbf{B}$ be a degenerate $(n+1)$-simplex 
in $\mathbb{R}^n$ with vertices $\{B_1, \dots, B_{n+2}\}$.
Then 
\begin{equation}
\label{equation_A_B}
\overrightarrow{B_{n+2}B_i}\cdot\overrightarrow{A_{n+2}A_j}=c.
\end{equation}
If $i,j,k$ and $n+2$ are distinct, in (\ref{equation_A_B})
replace $j$ with $k$ and subtract from it, then 
$\overrightarrow{B_{n+2}B_i}\cdot\overrightarrow{A_jA_k}=0$.
Similarly for distinct $i,j,k,l$, we have 
$\overrightarrow{B_iB_l}\cdot\overrightarrow{A_jA_k}=0$,
which verifies that $\mathbf{B}$ is a dual of $\mathbf{A}$.
By the construction of $\mathbf{B}$, it is not hard to observe that $\mathbf{B}$ 
is also unique up to similarity.

If we denote by $E_i$ the $n$-face of $\mathbf{B}$ that does not contain the vertex $B_i$,
then as a \emph{non-degenerate} simplex, $E_{n+2}$ is a \emph{dual}%
\footnote{This more conventional notion of dual is defined on any convex polytope $P$
in $\mathbb{R}^n$ by the following construction:
$P^{\ast}=\{y \in \mathbb{R}^n: x\cdot y \le c \quad\text{for all $x\in P$}\}$
with $c>0$ and a requirement that $P$ contains the origin $O$ in its interior.
But for a non-degenerate \emph{simplex} in $\mathbb{R}^n$,
if we are only concerned with the location of the vertices and not its \emph{interior},
then it can get away with this requirement (that $P$ contains the origin $O$)
as long as the origin $O$ is not on any hyperplane that contains an $(n-1)$-face.
}
of $F_{n+2}$ in $\mathbb{R}^n$ with respect to the origin $O$ (which is also $A_{n+2}$ and $B_{n+2}$ the same time),
a notion that in fact induces the notion of dual of a \emph{degenerate} simplex in $\mathbb{R}^n$.

\subsection{Properties of the characteristic polynomial}
\label{section_characteristic_polynomial}

Recall that the \emph{characteristic polynomial} of $(\mathbf{A},\alpha)$ is defined by
\[f(x)=\sum_{i=0}^{n+1}(-1)^i c_i(\alpha^i)x^{n+1-i}.
\]
For the Euclidean case, we showed in \cite[Theorem~3.4]{Zhang:rigidity} 
that $f(x)$ has one zero and $n$ non-zero real roots,
thus $c_{n+1}(\alpha^{n+1})$ is always 0.
But $c_{n+1}(\alpha^{n+1})$ is non-zero for a special spherical case (See Remark~\ref{remark_invariant_n_plus_1}),
so for the generality of $f(x)$,
we keep $(-1)^{n+1} c_{n+1}(\alpha^{n+1})$ as the constant term of $f(x)$.

For the rest of this section we consider the Euclidean case only.

For a $k$-simplex $F$ and two points $P$ and $Q$ in $\mathbb{R}^d$,
for convenience of computation,
we introduce a new notation $d_{\sss F}(P,Q)$.

\begin{definition}
\label{definition_partial_derivative_d}
For a $k$-simplex $F$ in $\mathbb{R}^d$,
define $d_{\sss F}(P,Q)$ by $k!V_k(F)\,g_{\sss F}(P,Q)$,
where $g_{\sss F}(P,Q)$ is defined in Definition~\ref{definition_partial_derivative_g_theta}.
Also set $d_{\sss\varnothing}(P,Q)=1$.
\end{definition}

\begin{remark}
Unlike the definition of $g_{\sss F}$ where $F$ need be non-degenerate,
here $d_{\sss F}$ is well defined when $F$ is degenerate,
and $P$ and $Q$ can be \emph{any} points as well. See below.
\end{remark}

If the vertices of $F$ are $P_1,\dots,P_{k+1}$,
then by \cite[(3.2), (2.5)]{Zhang:rigidity} we have

\begin{equation}
\label{equation_prod_r_matrix}
d_{\sss F}(P,Q) =
\det(\overrightarrow{PP_i}\cdot\overrightarrow{QP_j})_{1\leq i,j\leq k+1}.
\end{equation}
By Theorem~\ref{theorem_invariant_r_s_h} (\ref{equation_invariant_r_s_h})
and Definition~\ref{definition_k_stress_alpha}, 
for any choice of $P$ and $Q$, we have
\begin{align*}
c_{k+1}(\alpha^{k+1})
&=\sum_{F\subset \mathbf{A},\dim(F)=k}\alpha^{k+1}(F)\,g_{\sss F}(P,Q)     \\
&=\sum_{F\subset \mathbf{A},\dim(F)=k}
\left(\prod_{A_s\in F}\alpha_s\right) k!V_k(F)\,g_{\sss F}(P,Q)              \\
&=\sum_{F\subset \mathbf{A},\dim(F)=k}
\left(\prod_{A_s\in F}\alpha_s\right) d_{\sss F}(P,Q),
\end{align*}
and thus
\begin{equation}
\label{equaition_invariant_r_s_h_A_det}
c_{k+1}(\alpha^{k+1})=\sum_{F\subset \mathbf{A},\dim(F)=k}
\left(\prod_{A_s\in F}\alpha_s \right)
\det(\overrightarrow{PA_s}\cdot\overrightarrow{QA_t})_{A_s,A_t\in F}.
\end{equation}

Without loss of generality, we use the coordinate of $\mathbb{R}^n$ for $\mathbf{A}$ in the following.
Let $C_1$ be an $n\times n$ matrix whose $i$-th row is vector $\overrightarrow{A_{n+1}A_i}$ for $i\le n$,
$C_2$ be an $n\times n$ matrix whose $i$-th row is vector $\overrightarrow{A_{n+2}A_i}$,
and $D_1=\diag(\alpha_1,\dots,\alpha_n)$ be a diagonal matrix.

\begin{lemma}
\label{lemma_characteristic_polynomial_A} 
The characteristic polynomials of both matrix $C_1C_2^{T}D_1$ and $C_2^{T}D_1C_1$ are $f(x)/x$.
\end{lemma}

\begin{proof}
The coefficient of $x^{n-k}$ in the characteristic
polynomial of $C_1C_2^{T}D_1$ is $(-1)^k$ times the sum of
all principal minors of $C_1C_2^{T}D_1$ of order $k$,
which can be shown to be $(-1)^{k}c_k(\alpha^k)$
by choosing $P=A_{n+1}$ and $Q=A_{n+2}$
in (\ref{equaition_invariant_r_s_h_A_det}).
Also using the fact that $c_{n+1}(\alpha^{n+1})=0$,
thus $f(x)/x$ is the characteristic polynomial of $C_1C_2^{T}D_1$.
As $C_1$ and  $C_2^{T}D_1$ are two square matrices,
then the characteristic polynomials of $C_1C_2^{T}D_1$ and $C_2^{T}D_1C_1$ coincide.
Thus $f(x)/x$ is also the characteristic polynomial of $C_2^{T}D_1C_1$.
\end{proof}

\subsection{Proof of Theorem~\ref{theorem_dual_characteristic_polynomial}}

We continue to use the same notations from Section~\ref{section_characteristic_polynomial}.

Recall that if $\mathbf{B}$ is a dual of $\mathbf{A}$ in $\mathbb{R}^n$, 
then similarly to $\mathbf{A}$,
we obtain a sequence of invariants 
$c_0(\beta^0), \dots, c_{n+1}(\beta^{n+1})$ for $\mathbf{B}$, where $\beta$ is a 1-stress on $\mathbf{B}$,
and $\beta^k$ is induced as a $k$-stress on $\mathbf{B}$. 
The characteristic polynomial of $(\mathbf{B},\beta)$ is similarly defined by
\[g(x)=\sum_{i=0}^{n+1}(-1)^i c_i(\beta^i)x^{n+1-i}.
\]
Notice that $c_{n+1}(\beta^{n+1})=0$ as well.

Let $E_1$ be an $n\times n$ matrix whose $i$-th row is vector $\overrightarrow{B_{n+1}B_i}$ for $i\le n$,
$E_2$ be an $n\times n$ matrix whose $i$-th row is vector $\overrightarrow{B_{n+2}B_i}$,
and $D_2=\diag(\beta_1,\dots,\beta_n)$ be a diagonal matrix.
Then similarly to $f(x)$ (Lemma~\ref{lemma_characteristic_polynomial_A}), for $g(x)$ we have

\begin{lemma}
\label{lemma_characteristic_polynomial_B}
The characteristic polynomials of both matrix $E_1E_2^{T}D_2$ and $E_2^{T}D_2E_1$  are $g(x)/x$.
\end{lemma}

For an $n\times n$ matrix $A$ with non-zero determinant, the eigenvalues of the inverse matrix $A^{-1}$
are the same as the inverse of the eigenvalues of $A$.
So by Lemma~\ref{lemma_characteristic_polynomial_A} and \ref{lemma_characteristic_polynomial_B},
if we can show that the product of $C_2^{T}D_1C_1$ and $E_2^{T}D_2E_1$ is a multiple of the identity matrix $I_n$,
then we prove Theorem~\ref{theorem_dual_characteristic_polynomial}.
This is what we plan to do next.

As $\mathbf{B}$ is a dual of $\mathbf{A}$, then
$\overrightarrow{A_iA_j}\cdot\overrightarrow{B_kB_l}=0$ for all distinct $i,j,k,l$. 
Then
\[\overrightarrow{A_iA_j}\cdot\overrightarrow{B_iB_k}
=\overrightarrow{A_iA_j}\cdot\overrightarrow{B_iB_l}.
\]
So for a fixed $i$, $\overrightarrow{A_iA_j}\cdot\overrightarrow{B_iB_k}$ is independent of $j$ and $k$
as long as $i,j,k$ are distinct. We denote it by $r_i$. 
We will show that $\alpha_i r_i$ is independent of $i$.

Now consider the product of $\sum_{i=1}^{n+1}\frac{1}{r_i}\cdot\overrightarrow{A_{n+2}A_i}$
and $\overrightarrow{B_{n+1}B_j}$ for a fixed $j$ with $j\le n$.
There are only two non-zero terms left, one is 
$\frac{1}{r_j}\overrightarrow{A_{n+2}A_j}\cdot\overrightarrow{B_{n+1}B_j}$
that is equal to $\frac{r_j}{r_j}=1$, and the other is
$\frac{1}{r_{n+1}}\overrightarrow{A_{n+2}A_{n+1}}\cdot\overrightarrow{B_{n+1}B_j}$
that is equal to $\frac{-r_{n+1}}{r_{n+1}}=-1$.
So they cancel each other out. As $\overrightarrow{B_{n+1}B_j}$ with $j\le n$ are
$n$ linearly independent vectors in $\mathbb{R}^n$, 
so $\sum_{i=1}^{n+1}\frac{1}{r_i}\cdot\overrightarrow{A_{n+2}A_i}$ must be 0.
As $\sum_{i=1}^{n+1}\alpha_i\cdot\overrightarrow{A_{n+2}A_i}=0$ and the coefficients 
are unique up to a constant factor, 
so $1/r_i$ is proportional to $\alpha_i$ 
and thus $\alpha_i r_i$ is independent of $i$ for $i\le n+1$.
As $n\ge 1$, by symmetry, $\alpha_i r_i$ is independent of $i$ for $i\le n+2$ as well.

Also by symmetry, $\beta_i r_i$ is independent of $i$ for $i\le n+2$. 
So $\alpha$ and $\beta$ only differ by a constant factor,
and numerically we can set $\beta=\alpha$.

Now consider the matrix 
$C_1E_2^{T}=(\overrightarrow{A_{n+1}A_i}\cdot\overrightarrow{B_{n+2}B_j})_{1\le i,j\le n}$,
which by earlier argument is a diagonal matrix $\diag(r_1,\dots,r_n)$.
Similarly $E_1C_2^{T}$ is also $\diag(r_1,\dots,r_n)$ as well.
Since both $\alpha_i r_i$ and $\beta_i r_i$ are independent of $i$ for $i\le n+2$, 
so by putting what know together, we have
\[D_1(C_1E_2^{T})D_2(E_1C_2^{T})=c\cdot I_n
\]
for some constant $c$, where $I_n$ is the $n\times n$ identity matrix.
As the right side is $c\cdot I_n$, 
so on the left side of the formula
we can move $C_2^{T}$ from the end to the front 
and regroup the matrices 
without changing the value.

\begin{lemma}
\label{lemma_matrix_product_identity}
We have $(C_2^{T}D_1C_1)(E_2^{T}D_2E_1)=c\cdot I_n$ for some constant $c$.
\end{lemma}

By Lemma~\ref{lemma_characteristic_polynomial_A} and \ref{lemma_characteristic_polynomial_B},
$f(x)/x$ and $g(x)/x$ are the characteristic polynomials of 
$C_2^{T}D_1C_1$ and $E_2^{T}D_2E_1$ respectively.
Then by Lemma~\ref{lemma_matrix_product_identity},
we prove Theorem~\ref{theorem_dual_characteristic_polynomial}.

\begin{theorem}
\emph{(Theorem~\ref{theorem_dual_characteristic_polynomial})}
If $\mathbf{B}$ is a dual of $\mathbf{A}$ in $\mathbb{R}^n$
and the non-zero roots of $f(x)$ are $\{\lambda_1,\dots,\lambda_n\}$,
then the non-zero roots of $g(x)$ are $\{c/\lambda_1,\dots,c/\lambda_n\}$
for some constant $c$.
\end{theorem}

\subsection{Geometric interpretation of $c_{n-1}(\alpha^{n-1})=0$}
\label{section_geometric_interpretation_1}

As Theorem~\ref{theorem_dual_characteristic_polynomial} shows, 
if $\mathbf{B}$ is a dual of $\mathbf{A}$ in $\mathbb{R}^n$,
then for any $i$ with $0\le i\le n$,
$c_{n-i}(\alpha^{n-i})=0$ if and only if $c_i(\beta^i)=0$,
where $\beta$ is a 1-stress on $\mathbf{B}$,
and $\beta^k$ is induced as a $k$-stress on $\mathbf{B}$. 
Particularly $c_{n-1}(\alpha^{n-1})=0$ if and only if $c_1(\beta^1)=0$,
so for the Euclidean case interpreting $c_{n-1}(\alpha^{n-1})=0$
is the same as interpreting $c_1(\beta^1)=0$.

With a switch of notation between $c_1(\beta^1)$ and $c_1(\alpha^1)$,
here we provide a more general geometric interpretation of $c_1(\alpha^1)=0$
for not only the Euclidean but also the non-Euclidean cases.

\begin{proposition}
\emph{(\cite[Proposition~2.21]{Zhang:rigidity})}
\label{proposition_common_sphere_hyperplane}
For the spherical (resp. hyperbolic) case,
$c_1(\alpha^1)=0$ if and only if
$A_1$, \dots, $A_{n+2}$ are affinely dependent
in $\mathbb{R}^{n+1}$ (resp. $\mathbb{R}^{n,1}$).
For the Euclidean case,
$c_1(\alpha^1)=0$ if and only if $A_1$, \dots, $A_{n+2}$ 
are lying on an $(n-1)$-dimensional sphere in $\mathbb{R}^{n}$.
\end{proposition}

As the proof is rather simple and provides some geometric intuition for the reader,
we repeat it here.

\begin{proof}
For the non-Euclidean case, by (\ref{equation_invariant_s_h_A})
we have $c_1(\alpha^1)=\kappa\cdot2\sum\alpha_{i}$.
Since $\sum\alpha_{i}A_{i}=0$,
so $c_1(\alpha^1)=0$ (now the same as $\sum\alpha_{i}=0$) 
if and only if $A_1$, \dots, $A_{n+2}$ are \emph{affinely} dependent
in $\mathbb{R}^{n+1}$ or $\mathbb{R}^{n,1}$.

For the Euclidean case,
let $O_1$ be the center of the $(n-1)$-dimensional sphere in
$\mathbb{R}^{n}$ that contains points $A_{2}$, \dots, $A_{n+2}$,
and $r$ be the radius.
By (\ref{equation_alpha_r_s_h}) we have $\sum\alpha_i=0$, 
then by choosing $P=Q=O_1$ in (\ref{equaition_invariant_r_s_h_A_det}),
we have
\[c_1(\alpha^1)=\sum\alpha_{i}\overrightarrow{O_1A_i}^2=\alpha_1(\overrightarrow{O_1A_1}^2-r^2).
\]
Therefore $c_1(\alpha^1)=0$ if and only if $A_1$ is on the sphere as well.
\end{proof}

For $n=2$, to illustrate Theorem~\ref{theorem_lifting},
by using Proposition~\ref{proposition_common_sphere_hyperplane}
we give a rather interesting example of ``four points on a circle'' below.

\begin{example}
\label{example_four_points}
In $\mathbb{R}^3$, given four points that are initially in convex position on a plane.
If we allow the four points to vary smoothly in $\mathbb{R}^3$ but constrain them to preserve 
$\sum_{X_1}V_n(F_i)=\sum_{X_2}V_n(F_i)$ (see Question~\ref{question_lifting})
during any motion,
then the four points have to be confined to a plane first
until they move on to a common circle where $c_1(\alpha^1)=0$,
and only from this circle they can be \emph{lifted} 
from $\mathbb{R}^2$ to form a \emph{non}-degenerate 3-simplex in $\mathbb{R}^3$.
For the non-Euclidean case, the critical position when the four points can be lifted 
to form a non-degenerate 3-simplex is when $c_1(\alpha^1)=0$ as well,
namely when the points are affinely dependent in $\mathbb{R}^3$ or $\mathbb{R}^{2,1}$.
Particularly for the spherical case, it is the same as the four points are on 
a small circle on $\mathbb{S}^2$.%
\footnote{For the hyperbolic case, if we use the \emph{upper half-space model} instead, 
then it is also the same as the four points are on a circle (or a straight line)
in the upper half-plane $\mathbb{H}$, but the circle or line can be either fully or partially in $\mathbb{H}$.
}
\end{example}

For $n=2$, we give some simple examples to show that when $c_1(\alpha^1)=0$,
the four points can indeed be lifted from $\mathbb{R}^2$ 
to form a non-degenerate 3-simplex in $\mathbb{R}^3$
while preserving $\sum_{X_1}V_n(F_i)=\sum_{X_2}V_n(F_i)$ during the motion.

\begin{example}
\label{example_rectangle}
In $\mathbb{R}^3$, given four points that are initially on the $xy$-plane
with coordinates $(\pm a,\pm b,0)$. As they form a rectangle, so they are on a circle
and therefore $c_1(\alpha^1)=0$.
Now fix one pair of diagonal points during the motion, and for the other pair of diagonal points, 
let the coordinates $x$ and $y$ be constants and coordinate $z$ be $t$ for $t\ge 0$.
So for any $t>0$, the four points form a non-degenerate 3-simplex with four congruent faces,
and thus $\sum_{X_1}V_n(F_i)=\sum_{X_2}V_n(F_i)$ is preserved during the motion.
\end{example}

\begin{example}
\label{example_trapezoid}
In $\mathbb{R}^3$, given four points that initially form an \emph{isosceles trapezoid} 
on a plane. The lengths of its legs satisfy $l_1=l_2$, and the diagonals satisfy $d_1=d_2$.
As the four points are on a circle and therefore $c_1(\alpha^1)=0$.
Now if we only require $l_1(t)=l_2(t)$ and $d_1(t)=d_2(t)$
during the motion in $\mathbb{R}^3$, with no requirements
for the bases, then for $t>0$ the four points form a non-degenerate 3-simplex
whose two faces in $X_1$ are congruent to the two faces in $X_2$ respectively.
Thus $\sum_{X_1}V_n(F_i)=\sum_{X_2}V_n(F_i)$ is preserved during the motion.
This construction can also be similarly extended to the non-Euclidean case
for a quadrilateral that satisfies $l_1=l_2$ (legs) and $d_1=d_2$ (diagonals);
also notice that the notion of ``parallel'' no longer applies to the pair of bases,
but it is also not needed for the construction.
\end{example}

In fact, Example~\ref{example_rectangle} is a special case of 
Example~\ref{example_trapezoid}.

\subsection{Proof of Theorem~\ref{theorem_dual_sphere}}

By Theorem~\ref{theorem_dual_characteristic_polynomial},
for the Euclidean case
$c_{n-1}(\alpha^{n-1})=0$ if and only if $c_1(\beta^1)=0$.
Then by applying Proposition~\ref{proposition_common_sphere_hyperplane}
to $\mathbf{B}$, we prove Theorem~\ref{theorem_dual_sphere}.

\begin{theorem}
\emph{(Theorem~\ref{theorem_dual_sphere})}
If $\mathbf{B}$ is a dual of $\mathbf{A}$ in $\mathbb{R}^n$,
then $c_{n-1}(\alpha^{n-1})=0$ if and only if  $c_1(\beta^1)=0$,
which is also equivalent to the vertices of $\mathbf{B}$ lying on 
an $(n-1)$-dimensional sphere in $\mathbb{R}^n$.
\end{theorem}

\section{An alternative proof of Theorem~\ref{theorem_lifting_dual_sphere}}
\label{section_alternative_proof}

By combining Theorem~\ref{theorem_lifting} 
and \ref{theorem_dual_sphere} together, we provide a different formulation 
to answer Question~\ref{question_lifting} for the Euclidean case,
with a slightly stronger statement that includes continuous motion as well.

\begin{theorem}
\emph{(Theorem~\ref{theorem_lifting_dual_sphere})}
Let $\mathbf{B}$ be a dual of $\mathbf{A}$ in $\mathbb{R}^n$
and assume the vertices of $\mathbf{B}$ not lying on an $(n-1)$-dimensional sphere in $\mathbb{R}^n$.
If $\mathbf{A}$ varies continuously in $\mathbb{R}^d$ with a single volume constraint that
$\sum_{X_1}V_n(F_i)-\sum_{X_2}V_n(F_i)$ is preserved as 0,
then the vertices of $\mathbf{A}$
are confined to a lower dimensional $\mathbb{R}^n$ for any sufficiently small motion.
\end{theorem}

Recall that we used $k$-stress to prove Theorem~\ref{theorem_lifting} and matrix theory
(by proving Theorem~\ref{theorem_dual_characteristic_polynomial} first)
to prove Theorem~\ref{theorem_dual_sphere}.
While Theorem~\ref{theorem_lifting_dual_sphere}
is simply a statement that combines Theorem~\ref{theorem_lifting} 
and \ref{theorem_dual_sphere} together, somewhat surprisingly,
an alternative elementary proof of Theorem~\ref{theorem_lifting_dual_sphere} below
bypasses both Theorem~\ref{theorem_lifting} and \ref{theorem_dual_sphere}
and uses neither $k$-stress nor matrix theory,
and is valid for continuous motion as well.
Again, for the non-Euclidean case, we lack a similarly simple statement/proof,
thus making $k$-stress still the main tool to solve Question~\ref{question_lifting}.
Denote by $\mathbf{A}(t)$ the continuous motion (not necessarily smooth) of $\mathbf{A}$
and $\mathbf{A}(0)=\mathbf{A}$.

\begin{proof}
Assume the vertices of $\mathbf{A}(t)$ are lifted from $\mathbb{R}^n$ to form a non-degenerate 
$(n+1)$-simplex for small $t>0$. Also assume $d=n+1$, 
and $\mathbf{A}(0)$ is in a hyperplane  in $\mathbb{R}^{n+1}$ whose $(n+1)$-th coordinate is 0.
For an $n$-face $F_i(t)$ of $\mathbf{A}(t)$, let $u_i(t)$ be the \emph{outward} unit normal 
to $\mathbf{A}(t)$ at $F_i(t)$ if $F_i(t)$ is in $X_1$, and be the \emph{inward} unit normal
at $F_i(t)$ if $F_i(t)$ is in $X_2$. 
If we treat $u_i(t)$ as a unit vector pointing from the origin $O$ to a point $B_i(t)$ in $\mathbb{R}^{n+1}$, 
then by the \emph{Minkowski relation} for areas of facets of a Euclidean polytope, we have
\begin{equation}
\label{equation_minkowski_relation}
\sum_{X_1}V_n(F_i(t))B_i(t)-\sum_{X_2}V_n(F_i(t))B_i(t)=0.
\end{equation}

Let $N$ be the north pole with coordinates $(0,\dots,0,1)$,
then without loss of generality we assume $B_i(t)$ is in the neighborhood of $N$ for every $i$.
For a fixed $t>0$, as $\sum_{X_1}V_n(F_i(t))-\sum_{X_2}V_n(F_i(t))=0$,
so by (\ref{equation_minkowski_relation}) all $B_i(t)$ are \emph{affinely} dependent in $\mathbb{R}^{n+1}$.
Then all $B_i(t)$ are lying on the intersection of the unit $n$-sphere 
and a hyperplane in $\mathbb{R}^{n+1}$,
thus all $B_i(t)$ are lying on an $(n-1)$-sphere. 
Also notice that as $\overrightarrow{B_i(t)B_j(t)}=u_j(t)-u_i(t)$, so for all distinct $i,j,k,l$,
\begin{equation}
\label{equation_A_B_product_t}
\overrightarrow{B_i(t)B_j(t)}\cdot\overrightarrow{A_k(t)A_l(t)}=0.
\end{equation}

Denote by $\mathbf{B}(t)$ the collection of all points $B_i(t)$ for a fixed $t$.
Notationally, $\mathbf{B}(0)$ is not $\mathbf{B}$ but is a set of $n+2$ points that collapses to a single point $N$.
When $t\rightarrow 0$ (but not including $t=0$),
applying the facts that (1) $\mathbf{B}(t)$ is \emph{approximately} on the tangent space at $N$ 
(of the unit sphere) that is parallel to the hyperplane that contains $\mathbf{A}$,
(2) formula (\ref{equation_A_B_product_t}), and
(3) the uniqueness of $\mathbf{B}$ as a dual of $\mathbf{A}$ up to similarity,
we show that up to a proper scaling of $\mathbf{B}(t)$ for all $t>0$,
$\mathbf{B}(t)$ converges to a shape that is similar to $\mathbf{B}$.
Also notice that this property is independent of the path of $\mathbf{A}(t)$.

But as all $B_i(t)$ are lying on an $(n-1)$-sphere,
this contradicts the assumption that the vertices of $\mathbf{B}$ not lying on a 
$(n-1)$-sphere. Thus the vertices of $\mathbf{A}(t)$ 
are confined to a lower dimensional $\mathbb{R}^n$ for small $t\ge 0$.
\end{proof}

\begin{remark}
\label{remark_smooth_continuous}
The reader may notice that the proof above does not require that $\mathbf{A}(t)$ is \emph{smooth},
and is valid for continuous motion as well.
So combined with Theorem~\ref{theorem_dual_sphere}, 
it also makes Theorem~\ref{theorem_lifting} valid for continuous motion for the Euclidean case
without using more advanced tools.
\end{remark}

\begin{remark}
If the vertices of $\mathbf{B}$ are lying on an $(n-1)$-sphere, then in the proof above
by choosing a proper $\mathbf{B}(t)$ 
such that $\mathbf{B}(t)$ is similar to $\mathbf{B}$ for all $t>0$,
we can, inversely, explicitly construct a non-degenerate $\mathbf{A}(t)$ that preserves 
$\sum_{X_1}V_n(F_i)-\sum_{X_2}V_n(F_i)=0$ during the motion.
\end{remark}

\noindent
{\bf Acknowledgements:}
One of the main results Theorem~\ref{theorem_dual_characteristic_polynomial} 
was obtained during the author's PhD thesis work at M.I.T..
I would like to thank Professor Kleitman and Professor Stanley for their helpful discussions.
I would also like to thank Wei Luo and a referee for carefully reading the manuscript and
making many suggestions.

{\footnotesize
\bibliographystyle{abbrv}  
\bibliography{lifting_arxiv}   

\begin{thebibliography}{10}

\bibitem{Alexandrov:flexible}
V.~Alexandrov.
\newblock An example of a flexible polyhedron with nonconstant volume in the
  spherical space.
\newblock {\em Beitr. Algebra Geom.}, 38(1):11--18, 1997.

\bibitem{Connelly:counterexample}
R.~Connelly.
\newblock A counterexample to the rigidity conjecture for polyhedra.
\newblock {\em Publ. Math. Inst. Hautes \'{E}tud. Sci.}, 47:333--338, 1977.

\bibitem{ConnellySabitovWalz}
R.~Connelly, I.~Sabitov, and A.~Walz.
\newblock The bellows conjecture.
\newblock {\em Beitr. Algebra Geom.}, 38(1):1--10, 1997.

\bibitem{Kalai:rigidity}
G.~Kalai.
\newblock Rigidity and the lower bound theorem. {I}.
\newblock {\em Invent. Math.}, 88(1):125--151, 1987.

\bibitem{Lee:stress}
C.~W. Lee.
\newblock P.{L}.-spheres, convex polytopes, and stress.
\newblock {\em Discrete Comput. Geom.}, 15(4):389--421, 1996.

\bibitem{McMullen:weights}
P.~McMullen.
\newblock Weights on polytopes.
\newblock {\em Discrete Comput. Geom.}, 15(4):363--388, 1996.

\bibitem{Milnor:Schlafli}
J.~Milnor.
\newblock The {S}chl{\"a}fli differential equality.
\newblock In {\em Collected Papers, vol. 1}. Publish or Perish, New York, 1994.

\bibitem{RivinSchlenker}
I.~Rivin and J.-M. Schlenker.
\newblock The {S}chl{\"a}fli formula in {Einstein} manifolds with boundary.
\newblock {\em Electron. Res. Announc. Am. Math. Soc.}, 5:18--23, 1999.

\bibitem{Rybnikov:stresses}
K.~Rybnikov.
\newblock Stresses and liftings of cell complexes.
\newblock {\em Discrete Comput. Geom.}, 21(4):481--517, 1999.

\bibitem{Sabitov:invariance}
I.~Sabitov.
\newblock On the problem of invariance of the volume of a flexible polyhedron.
\newblock {\em Russ. Math. Surv.}, 50(2):451--452, 1995.

\bibitem{Stanley:number}
R.~P. Stanley.
\newblock The number of faces of a simplicial convex polytope.
\newblock {\em Adv. Math.}, 35(3):236--238, 1980.

\bibitem{Suarez:deSitter}
E.~Su{\'a}rez-Peir{\'o}.
\newblock A {S}chl{\"a}fli differential formula for simplices in
  semi-riemannian hyperquadrics, {Gauss-Bonnet} formulas for simplices in the
  de {Sitter} sphere and the dual volume of a hyperbolic simplex.
\newblock {\em Pac. J. Math.}, 194(1):229--255, 2000.

\bibitem{TWW}
T.-S. Tay, N.~White, and W.~Whiteley.
\newblock Skeletal rigidity of simplicial complexes, {I}.
\newblock {\em Eur. J. Comb.}, 16(4):381--403, 1995.

\bibitem{Zhang:rigidity}
L.~Zhang.
\newblock Rigidity and volume preserving deformation on degenerate simplices.
\newblock {\em Discrete Comput. Geom.}, 60(4):909--937, 2018.

\end{thebibliography}

%
%
}

\end{document}